\documentclass[12pt]{article}
\usepackage[dvips]{epsfig}
\usepackage{graphics}
\usepackage{amsmath}
\usepackage{amsfonts}
\usepackage{amssymb}
\usepackage{amsxtra}
\usepackage[mathscr]{eucal}
\usepackage{amsthm}
\usepackage{bm}
\usepackage{enumerate}
\usepackage[mathscr]{eucal}
\usepackage{inputenc}
\usepackage[T2A]{fontenc}
\DeclareSymbolFont{rsfs}{U}{rsfs}{m}{n}
\DeclareSymbolFontAlphabet{\mathscrn}{rsfs} \theoremstyle
{plain}
\newtheorem{theorem}{Theorem}
\newtheorem{lemma}{Lemma}
\newtheorem{remark}{Remark}
 \numberwithin{equation}{section}

\allowdisplaybreaks[4]

\begin{document}
\title{\textbf{Eigenfunctions of the Cosine and Sine Transforms}}
\author{\textbf{Victor Katsnelson}}
\date{\ }%
 \maketitle
 \vspace*{-6.0ex}
 \begin{center}
 \textit{The Weizmann  Institute of Science\\
 Rehovot, 76100, Israel}
 \end{center}
 \begin{center}{\small
 E-mail: \texttt{victor.katsnelson@weizmann.ac.il; victorkatsnelson@gmail.com}
 }
 \end{center}
{\abstract\footnotesize{}A description of eigensubspaces of the cosine and sine operators is presented.
 The spectrum of each of these two operator consists of two eigenvalues \(1,\,-1\) and
 their eigensubspaces are infinite--dimensional. There are many possible bases for these subspaces, but most popular are bases constructed from the Hermite functions.
We present other "bases" which are not discrete orthogonal sequences of vectors, but continuous orthogonal chains of vectors.  Our work can be considered
 a continuation and further development of results in \textit{Self-reciprocal functions}  by Hardy and Titchmarsh:
 Quarterly Journ. of Math.
(Oxford Ser.) \textbf{1} (1930).
 }\\


\vspace{3.0ex}

 \textbf{1.}\hspace{1.0ex} The cosine transform \(\boldsymbol{\mathscr{C}}\) and the sine transform
 \(\boldsymbol{\mathscr{S}}\) are defined by formulas
 \begin{subequations}
\label{CoSiTr}
\begin{align}
\label{CosTr}
(\boldsymbol{\mathscr{C}}x)(t)
=&{\textstyle\sqrt{\frac{2}{\pi}}}%
{\displaystyle\int\limits_{\mathbb{R_{+}}}\cos(t\xi)\,x(\xi)\,d\xi}, \quad t\in\mathbb{R}_{+}\,,\\
\label{SinTr}
(\boldsymbol{\mathscr{S}}x)(t)=
&{\textstyle\sqrt{\frac{2}{\pi}}}\int\limits_{\mathbb{R_{+}}}\sin(t\xi)\,x(\xi)\,d\xi,
\quad t\in\mathbb{R}_{+}\,,
\end{align}
\end{subequations}
where \(\mathbb{R}_{+}\) is the positive half-axis,
\(\mathbb{R}_{+}=\lbrace{}t\in\mathbb{R}:\,t>0\rbrace\).\\
For \(x\in{}L^1(\mathbb{R}_{+})\), the integrals in
\eqref{CoSiTr} are well defined as Lebesgue integrals. If
\(x(t)\in{}L^2(\mathbb{R}_{+})\cap{}L^1(\mathbb{R}_{+})\), then the Parseval equalities hold:
\begin{subequations}
\label{PaCS}
\begin{align}
\label{PaC}
\int\limits_{\mathbb{R}_{+}}|(\boldsymbol{\mathscr{C}}x)(t)|^2dt=
&\int\limits_{\mathbb{R}_{+}}|x(t)|^2dt,\\
\label{PaS}
\int\limits_{\mathbb{R}_{+}}|(\boldsymbol{\mathscr{S}}x)(t)|^2dt=
&\int\limits_{\mathbb{R}_{+}}|x(t)|^2dt.
\end{align}
\end{subequations}
Thus, the transforms \(\boldsymbol{\mathscr{C}}\) and  \(\boldsymbol{\mathscr{S}}\)
can both be considered as linear operators defined on the
linear manifold \(L^{1}(\mathbb{R}_{+})\cap{}L^{2}(\mathbb{R}_{+})\) of the Hilbert space \(L^{2}(\mathbb{R}_{+})\),
mapping this linear manifold into \(L^{2}(\mathbb{R}_{+})\) \emph{isometrically}.
Since the set \(L^{1}(\mathbb{R}_{+})\cap{}L^{2}(\mathbb{R}_{+})\) is dense in \(L^{2}(\mathbb{R}_{+})\),   each of these operators can be extended to an operator
defined on the \emph{whole} space  \(L^{2}(\mathbb{R}_{+})\),  which maps  \(L^{2}(\mathbb{R}_{+})\)
into \(L^{2}(\mathbb{R}_{+})\) isometrically.
We retain the notation  \(\boldsymbol{\mathscr{C}}\) and \(\boldsymbol{\mathscr{S}}\)
 for the extended operators.
 In an even broader context, the transformation \eqref{CoSiTr} can be considered for
 those \(x\), for which the the integrals on the right--hand sides are meaningful.

Considered as operators in the Hilbert space \(L^{2}(\mathbb{R}_{+})\), the operators
 \(\boldsymbol{\mathscr{C}}\) and \(\boldsymbol{\mathscr{S}}\) are self-adjoint operators
 which satisfy the equalities
 \begin{equation}
\label{Inv}
\boldsymbol{\mathscr{C}}^{2}=\boldsymbol{\mathscr{I}},\quad \boldsymbol{\mathscr{S}}^{2}=\boldsymbol{\mathscr{I}},
\end{equation}
where \(\boldsymbol{\mathscr{I}}\) is the identity operator in \(L^2(\mathbb{R}_{+})\).
Each of the spectra \(\sigma(\boldsymbol{\mathscr{C}})\) and
\(\sigma(\boldsymbol{\mathscr{S}})\) of these operators
consist of two points: \(+1\) and \(-1\). By \(\mathcal{C}_{\lambda}\) and
\(\mathcal{S}_{\lambda}\) we denote the spectral subspaces of the
 operators  \(\boldsymbol{\mathscr{C}}\) and \(\boldsymbol{\mathscr{S}}\),
 respectively, corresponding to the points \(\lambda=1\) and  \(\lambda=-1\)
of their spectra. These spectral subspaces are eigensubspaces:
\begin{subequations}
\label{EigSub}
\begin{alignat}{2}
\label{EigSubc}
\mathcal{C}_{1}&=\{x\in{}L^2(\mathbb{R}_{+}): \boldsymbol{\mathscr{C}}x=x\},\quad &
\mathcal{C}_{-1}&=\{x\in{}L^2(\mathbb{R}_{+}): \boldsymbol{\mathscr{C}}x=-x\};\\
\label{EigSubs}
\mathcal{S}_{1}&=\{x\in{}L^2(\mathbb{R}_{+}): \boldsymbol{\mathscr{S}}x=x\},\quad &
\mathcal{S}_{-1}&=\{x\in{}L^2(\mathbb{R}_{+}): \boldsymbol{\mathscr{S}}x=-x\}.
\end{alignat}
\end{subequations}
Moreover, two orthogonal decompositions hold:
\begin{equation}
\label{OrDe}
L^2(\mathbb{R}_{+})=\mathcal{C}_{1}\oplus\mathcal{C}_{-1},\quad
L^2(\mathbb{R}_{+})=\mathcal{S}_{1}\oplus\mathcal{S}_{-1}\,.
\end{equation}
The spectra of the operators \(\boldsymbol{\mathscr{C}}\) and \(\boldsymbol{\mathscr{S}}\)
 are highly degenerated: the eigensubspaces \(\mathcal{C}_{\lambda}\) and
\(\mathcal{S}_{\lambda}\) are infinite--dimensional. Many bases are
 possible in these subspaces. The best known are the bases formed by
the Hermite functions \(h_{k}(t)\) restricted onto \(\mathbb{R}_{+}\).

The Hermite functions \(h_k(t)\) are defined as
 \begin{equation}
 \label{HF}
h_{k}(t)=e^{\frac{t^2}{2}}\,\dfrac{d^k\,(e^{-t^2})}{dt^k},\,
\,t\in\mathbb{R},\,\,\,\,\,k=0,\,1,\,2,\,\ldots\,\,.
\end{equation}
It is known that the system \(\lbrace{}h_{k}\rbrace_{k=0,1,2,\,\ldots\,\,}\)
forms an orthogonal basis in the Hilbert space \(L^2(\mathbb{R})\).
The properties of the Hermite functions \(h_k\) as eigenfunctions of
the Fourier transform was established by N.\,Wiener, \cite[Chapter 1]{1}.
 In \cite{1}, N.\,Wiener developed
\(L^2\)-theory of the Fourier transform which was based on these
properties of the Hermite functions.

The Hermite functions \(h_{k}\) are originally defined on the whole real
axis \(\mathbb{R}\). The restrictions \({h_k}_{|\mathbb{R}_{+}}\)
of the Hermite functions \(h_k\) onto \(\mathbb{R}_{+}\) are considered
as vectors of the Hilbert space \(L^2(\mathbb{R}_{+})\).
Each of two systems
\(\lbrace{h_{2k}}_{_{|\mathbb{R}_+}}\rbrace_{k=0,1,2,\,\ldots\,\,}\)
and \(\lbrace{h_{2k+1}}_{_{|\mathbb{R}_+}} {}\rbrace_{k=0,1,2,\,\ldots\,\,}\) \,\,is an orthogonal basis in \(L^2(\mathbb{R}_{+})\). The systems \(\big\lbrace{h_{4l}}_{_{|\mathbb{R}_+}}\big\rbrace_{l=0,1,2,\,\ldots}\),
\(\big\lbrace{h_{4l+2}}_{_{|\mathbb{R}_+}}\big\rbrace_{l=0,1,2,\,\ldots}\),
\(\big\lbrace{h_{4l+1}}_{_{|\mathbb{R}_+}}\big\rbrace_{l=0,1,2,\,\ldots}\), and
\(\big\lbrace{h_{4l+3}}_{_{|\mathbb{R}_+}}\big\rbrace_{l=0,1,2,\,\ldots}\) are orthogonal bases of the eigensubspaces
\(\mathcal{C}_{1}\), \(\mathcal{C}_{-1}\), \(\mathcal{S}_{1}\) and \(\mathcal{S}_{-1}\),
respectively.
 We present other "bases" which are not discrete orthogonal sequences of vectors, but continuous orthogonal chains of (generalized) vectors. This is the main goal of this paper.  Our work may be considered as a further development of results in \cite{2} by Hardy and Titchmarsh. (The contents of \cite{2} and \cite{3} were reproduced in the book \cite{4}.)

\noindent
 \textbf{2.}\hspace{1.0ex} First we discuss eigenfunctions of the transforms
\(\boldsymbol{\mathscr{C}}\) and \(\boldsymbol{\mathscr{S}}\) in the broad sense.
These transforms are of the form \(x\to\boldsymbol{\mathscr{K}}x\), where
\begin{equation}
\label{KDoP}
(\boldsymbol{\mathscr{K}}x)(t)=\int\limits_{\mathbb{R}_{+}}k(t\xi)x(\xi)d\xi,
\end{equation}
and \(k\) is a function of \emph{one} variable defined on \(\mathbb{R}_{+}\).
(It should be mentioned that some operational calculus related to operators of the form \eqref{KDoP} was developed in \cite{5}.)
\begin{remark}
\label{RePr}
If the integral \eqref{KDoP} does not exist as a Lebesgue integral,
i.e. the function \(k(t\xi)x(\xi)\) of the
variable \(\xi\) is not summable, then a meaning may be attached to
the integral  \eqref{KDoP}  by means of some regularization procedure.  We use the regularization procedure
\begin{equation}
\label{OCRP1}
\int\limits_{\mathbb{R}_{+}}k(t\xi)x(\xi)d\xi=\lim\limits_{\varepsilon\to+0}
\int\limits_{\mathbb{R}_{+}}e^{-\varepsilon{}\xi}k(t\xi)x(\xi)d\xi,\\
\end{equation}
and the regularization procedure
\begin{equation}
\label{OCRP2}
\int\limits_{\mathbb{R}_{+}}k(t\xi)x(\xi)d\xi=\lim\limits_{R\to+\infty}
\int\limits_{0}^{R}k(t\xi)x(\xi)d\xi\,.
\end{equation}
\end{remark}{\ }\\

If for some \(a\in\mathbb{C}\) both integrals
\begin{equation}
\label{OA}
\int\limits_{\mathbb{R}_{+}}k(t\xi){\xi}^{-a}d\xi \quad \textup{and}\quad \int\limits_{\mathbb{R}_{+}}k(t\xi){\xi}^{a-1}d\xi
\end{equation}
 have a meaning for every positive \(t\),
then, changing the variable \mbox{\(t\xi\to\xi\)}, we obtain
\begin{equation}
\label{GeEiFu}
\boldsymbol{\mathscr{K}}\,t^{-a}=\varkappa(a)\,t^{a-1},\quad \boldsymbol{\mathscr{K}}\,t^{a-1}=\varkappa(1-a)\,t^{-a},
\end{equation}
where
\begin{equation}
\label{MaGeOp}
{\varkappa}(a)=\int\limits_{\mathbb{R}_{+}}k(\xi)\xi^{-a}d\xi,\quad
\varkappa(1-a)=\int\limits_{\mathbb{R}_{+}}k(\xi)\xi^{a-1}d\xi.
\end{equation}

The equalities \eqref{GeEiFu} mean that the subspace (two-dimensional if \(a\not=1/2\))
generated by the functions \(t^{-a}\) and \(t^{a-1}\)
is invariant with respect to the transformation \(\boldsymbol{\mathscr{K}}\) and that
the matrix of this operator in the basis \(t^{-a},t^{a-1}\) is:
\begin{math}
\left\|\begin{smallmatrix}
0&\varkappa(1-a)\\
\varkappa(a)&0
\end{smallmatrix}
\right\|.
\end{math}
Thus, assuming that \(\varkappa(a)\not=0,\,\varkappa(1-a)\not=0\), we obtain that the functions
\begin{equation}
\label{EiFuK}
\sqrt{\varkappa(1-a)}t^{-a}+\sqrt{\varkappa(a)}t^{a-1}\quad \textup{and} \quad \sqrt{\varkappa(1-a)}t^{-a}-\sqrt{\varkappa(a)}t^{a-1}
\end{equation}
are the eigenfunctions of the transform \(\boldsymbol{\mathscr{K}}\) corresponding to the
eigenvalues
\begin{equation}
\label{GeEiVa}
\lambda_{+}=\sqrt{\varkappa(a)\varkappa(1-a)}\quad \textup{and} \quad\lambda_{-}=-\sqrt{\varkappa(a)\varkappa(1-a)},
\end{equation}
 respectively.

 To find eigenfunctions  of the form
\eqref{EiFuK} for cosine and sine transforms \(\boldsymbol{\mathscr{C}}\) and \(\boldsymbol{\mathscr{S}}\), we have to calculate the constants \eqref{MaGeOp} corresponding to the functions
 \begin{equation}
 \label{FGKe}
 k_c(\tau)=\sqrt{\frac{2}{\pi}}\cos{}\tau\quad and  \quad{}k_s(\tau)=\sqrt{\frac{2}{\pi}}\sin{}\tau,
 \end{equation}
 which generate the kernels of these integral transforms. This is accomplished in the following\\

 \noindent
 \begin{lemma}%
\label{CCI}%
Let \(\zeta\) belong to the strip
\begin{math}%
0<\textup{Re}\,\zeta<1.
\end{math}%

\noindent
Then
\begin{enumerate}
\item[\textup{1}.] {\ } \\[-8.0ex]
\begin{subequations}
\label{CcI}
\begin{gather}
\label{CcI1}%
 \int\limits_{0}^{\infty}(\cos{}s)\,s^{\zeta-1}\,ds=
\Big(\cos\,\frac{\pi}{2}\zeta\Big)\,\Gamma(\zeta)\,,\\[1.0ex]
 \label{CcI2}
 \int\limits_{0}^{\infty}(\sin{}s)\,s^{\zeta-1}\,ds=
\Big(\sin\,\frac{\pi}{2}\zeta\Big)\,\Gamma(\zeta)\,,
\end{gather}
\end{subequations}
where \(\Gamma\) is the Euler Gamma-function
and the integrals in \eqref{CcI} are understood in the sense
\begin{multline*}
\int\limits_{0}^{\infty}
\bigg\lbrace
\begin{matrix}
\cos{}s\\
\sin{}s
\end{matrix}
\bigg\rbrace\, s^{\,\zeta-1}\,ds\\
=\lim_{R\to+\infty}\int\limits_{0}^{R}\bigg\lbrace
\begin{matrix}
\cos{}s\\
\sin{}s
\end{matrix}
\bigg\rbrace\, s^{\,\zeta-1}\,ds
=\lim_{\varepsilon\to+0}\int\limits_{0}^{+\infty}e^{-\varepsilon{}s}\bigg\lbrace
\begin{matrix}
\cos{}s\\
\sin{}s
\end{matrix}
\bigg\rbrace\, s^{\,\zeta-1}\,ds\,;
\end{multline*}
\item[\textup{2.}] The above limits  exist
uniformly with respect to \(\zeta\) from any fixed compact subset
of the strip \(0<\textup{Re}\,\zeta<1\).
\item[\textup{3.}] Given \(\delta>0\), then for any \(\zeta\) from the strip
\(\delta<\textup{Re}\,\zeta<1-\delta\) and for any \(R\in(0,+\infty)\),
\(\varepsilon\in(0,+\infty)\),
the estimates
\begin{subequations}
\label{Ewrta}
\begin{align}
\label{Ewrta1}
\bigg|\int\limits_{0}^{R}\bigg\lbrace
\begin{matrix}
\cos{}s\\
\sin{}s
\end{matrix}
\bigg\rbrace\,
s^{\zeta-1}ds\bigg|\leq{}C(\delta)%
e^{\frac{\pi}{2}|\textup{Im}\,\zeta|},\\
\label{Ewrta2}
\bigg|\int\limits_{0}^{+\infty}e^{-\varepsilon{}s}\bigg\lbrace
\begin{matrix}
\cos{}s\\
\sin{}s
\end{matrix}
\bigg\rbrace\,
s^{\zeta-1}ds\bigg|\leq{}C(\delta)%
e^{\frac{\pi}{2}|\textup{Im}\,\zeta|},
\end{align}
\end{subequations}
hold, where \(C(\delta)<\infty\) does not depend on \(\zeta\), \(R\), and \(\varepsilon\).
\end{enumerate}
\end{lemma}%

We omit proof of Lemma \ref{CCI}. This lemma can be proved by a standard method using integration in the complex plane.\\[-2ex]
\hbox to \linewidth{\hfil \(\square\)}\\

      According to Lemma \ref{CCI},  the integrals
\[\int\limits_{0}^{\infty}
\bigg\lbrace
\begin{matrix}
k_c(s)\\
k_s(s)
\end{matrix}
\bigg\rbrace\, s^{\,-a}\,ds
\quad\textup{and}\quad
\int\limits_{0}^{\infty}
\bigg\lbrace
\begin{matrix}
k_c(s)\\
k_s(s)
\end{matrix}
\bigg\rbrace\, s^{a-1}\,ds,
\]
where \(k_c\) and \(k_s\), \eqref{FGKe}, are the functions generating the kernels of the integral
 transformations \(\boldsymbol{\mathscr{C}}\) and \(\boldsymbol{\mathscr{S}}\),
exist for every \(a\) such that \mbox{\(0<\textup{Re}\,a<1\)}, or, amounting to the same,
\(0<\textup{Re}\,(1-a)<1\).\\
The constants
\(\varkappa_c(a)\) and \(\varkappa_c(1-a)\), corresponding to the function \(k_c(\tau)=\sqrt{\frac{2}{\pi}}\cos{}\tau\), are:
\begin{subequations}
\label{kapp}
\begin{equation}
\label{kappc}
\varkappa_c(a)=\sqrt{\frac{2}{\pi}}\sin\frac{\pi{}a}{2}\,\Gamma(1-a),\quad
\varkappa_c(1-a)=\sqrt{\frac{2}{\pi}}\cos\frac{\pi{}a}{2}\,\Gamma(a).
\end{equation}
The constants
\(\varkappa_s(a)\) and \(\varkappa_s(1-a)\), corresponding to the function \(k_s(\tau)=\sqrt{\frac{2}{\pi}}\sin{}\tau\), are:
\begin{equation}
\label{kapps}
\varkappa_s(a)=\sqrt{\frac{2}{\pi}}\cos\frac{\pi{}a}{2}\,\Gamma(1-a),\quad
\varkappa_s(1-a)=\sqrt{\frac{2}{\pi}}\sin\frac{\pi{}a}{2}\,\Gamma(a).
\end{equation}
\end{subequations}

\noindent
\textbf{3.}\hspace{1.0ex}
Later, we will have to transform the expression \eqref{kapp} for the
constants \(\varkappa_c\) and \(\varkappa_s\) using the following \\
\textsf{Identities for the Euler Gamma-function
 \(\Gamma(\zeta)\)}:
 \begin{subequations}
 \label{Ga}
\begin{alignat}{2}%
\label{Ga1}
\Gamma(\zeta+1)&=\zeta\Gamma(\zeta)\,,& &\quad\text{see\,\,\,%
\cite{6}\,,\,\,\,\textbf{12.12},}\\[1.5ex]
\label{Ga2}
\Gamma(\zeta)\Gamma(1-\zeta)&=\dfrac{\pi}{\sin{}\pi\zeta}\,,\,\,\,&
&
\quad\text{see\,\,\,\cite{6}\,,\,\,\,\textbf{12.14},}\\[1.5ex]
\label{Ga3}
\Gamma(\zeta)\Gamma\bigg(\zeta+\frac{1}{2}\bigg)&=2\sqrt{\pi}\,2^{-2\zeta}\Gamma(2\zeta),\,\,&
& \quad\text{see\,\,\,\cite{6}\,,\,\,\,\textbf{12.15}.}
\end{alignat}
\end{subequations}
\begin{lemma}
\label{Idka}
The following identities hold:
\begin{subequations}
\begin{gather}
\label{Gam1}%
 \sqrt{\frac{2}{\pi}}\,
\Big(\cos\frac{\pi}{2}\zeta\Big)\,\Gamma(\zeta)=%
2^{\,\zeta-\frac{1}{2}}%
\frac{\Gamma\big(\frac{\zeta}{2}\big)}{\Gamma\big(\frac{1}{2}-\frac{\zeta}{2}\big)}\,,\\
\label{Gam2} \sqrt{\frac{2}{\pi}}\,
\Big(\sin\frac{\pi}{2}\zeta\Big)\,\Gamma(\zeta)=2^{\,\zeta-\frac{1}{2}}%
\frac{\Gamma\big(\frac{1}{2}+\frac{\zeta}{2}\big)}{\Gamma\big(1-\frac{\zeta}{2}\big)}\,.
\end{gather}
\end{subequations}
\end{lemma}
\begin{proof}
From \eqref{Ga2} it follows that
\[\cos\frac{\pi}{2}\zeta=\frac{\pi}{\Gamma\big(\frac{1}{2}-\frac{\zeta}{2}\big)\,%
\Gamma\big(\frac{1}{2}+\frac{\zeta}{2}\big)}\,.\]%
 From \eqref{Ga3} it follows that
 \[\Gamma(\zeta)=
 \pi^{-\frac{1}{2}}\,\Gamma{\textstyle\big(\frac{\zeta}{2}\big)}\,%
 \Gamma{\textstyle\big(\frac{1}{2}+\frac{\zeta}{2}\big)}\,2^{\zeta-1}\,.\]
 Combining the last two formulas, we obtain \eqref{Gam1}.
 Combining the last formula with the formula
 \[\sin\frac{\pi}{2}\zeta=\frac{\pi}{\Gamma\big(\frac{\zeta}{2}\big)\,%
\Gamma\big(1-\frac{\zeta}{2}\big)}\,,\]
 we obtain \eqref{Gam2}.
\end{proof}
\begin{lemma}
\label{OthE}
The values \(\varkappa_c(a),\,\varkappa_c(1-a),\,\varkappa_s(a),\,\varkappa_s(1-a)\),
which appear as coefficients of linear combinations \eqref{EiFuK},
are
\begin{subequations}
\label{CLCE}
\begin{alignat}{2}
\label{CLCEc}
\varkappa_c(a)&=
2^{\frac{1}{2}-a}\frac{\Gamma(\frac{1}{2}-\frac{a}{2})}{\Gamma(\frac{a}{2})},
\quad&\quad
\varkappa_c(1-a)&=
2^{a-\frac{1}{2}}\frac{\Gamma(\frac{a}{2})}{\Gamma(\frac{1}{2}-\frac{a}{2})}\\
\label{CLCEs}
\varkappa_s(a)&=
2^{\frac{1}{2}-a}\frac{\Gamma(1-\frac{a}{2})}{\Gamma(\frac{1}{2}+\frac{a}{2})},
\quad&\quad
\varkappa_s(1-a)&=
2^{a-\frac{1}{2}}\frac{\Gamma(\frac{1}{2}+\frac{a}{2})}{\Gamma(1-\frac{a}{2})}.
\end{alignat}
\end{subequations}
\end{lemma} {\ }\\

\noindent
 \textbf{4.}\hspace{1.0ex}
From the expressions \eqref{CLCE} we see that
the products \(\varkappa_c(a)\varkappa_c(1-a)\) and \(\varkappa_s(a)\varkappa_s(1-a)\)
do not depend on \(a\):
\begin{equation*}
\varkappa_c(a)\varkappa_c(1-a)=1,\quad \varkappa_s(a)\varkappa_s(1-a)=1\,\quad 0<\textup{Re}\,a<1.
\end{equation*}
\begin{theorem}
\label{GeEig}
Let \(a\in\mathbb{C},\,0<\textup{Re}\,a<1\), \(a\not=\tfrac{1}{2}\), and
\(\varkappa_c(a),\,\varkappa_c(1-~a)\), \(\varkappa_s(a)\), \(\varkappa_s(1-a)\) be the values which appear
in \eqref{CLCE}.

\noindent
 Then:
\begin{enumerate}
\item[\textup{1.}]
 The functions
 \begin{subequations}
 \label{GECoT}
\begin{align}
\label{GECoTp}
E_{c}^{+}(t,a)&=\sqrt{\varkappa_c(1-a)}t^{-a}+\sqrt{\varkappa_c(a)}t^{a-1},\\
\label{GECoTm}
E_{c}^{-}(t,a)&=\sqrt{\varkappa_c(1-a)}t^{-a}-\sqrt{\varkappa_c(a)}t^{a-1},
\end{align}
\end{subequations}
of  variable \(t\in\mathbb{R}_{+}\)
are eigenfunctions \textup{(}in a broad sense\textup{)} of the cosine transform
\(\boldsymbol{\mathscr{C}}\) corresponding to the eigenvalues \(+1\) and \(-1\) respectively:
\begin{align*}
E_{c}^{+}(t,a)&=\phantom{-}
\lim_{R\to\infty}\sqrt{\tfrac{2}{\pi}}\int\limits_{0}^{R}\cos(t\xi)\,E_{c}^{+}(\xi,a)\,d\xi,\\
E_{c}^{-}(t,a)&=-
\lim_{R\to\infty}\sqrt{\tfrac{2}{\pi}}\int\limits_{0}^{R}\cos(t\xi)\,E_{c}^{-}(\xi,a)\,d\xi;
\end{align*}
 \item[\textup{2.}]
 The functions
  \begin{subequations}
  \label{GESiT}
\begin{align}
\label{GESiTp}
E_{s}^{+}(t,a)&=\sqrt{\varkappa_s(1-a)}t^{-a}+\sqrt{\varkappa_s(a)}t^{a-1},\\
\label{GESiTm}
E_{s}^{-}(t,a)&=\sqrt{\varkappa_s(1-a)}t^{-a}-\sqrt{\varkappa_s(a)}t^{a-1}
\end{align}
\end{subequations}
of  variable \(t\in\mathbb{R}_{+}\) are eigenfunctions \textup{(}in a broad sense\textup{)} of the sine\phantom{co} transform
\(\boldsymbol{\mathscr{S}}\) corresponding to the eigenvalues \(+1\) and \(-1\) respectively:
\begin{align}
E_{s}^{+}(t,a) &=\phantom{-}
\lim_{R\to\infty}\sqrt{\tfrac{2}{\pi}}\int\limits_{0}^{R}\sin(t\xi)\,E_{s}^{+}(\xi,a)\,d\xi,\\
E_{s}^{-}(t,a) &=-
\lim_{R\to\infty}\sqrt{\tfrac{2}{\pi}}\int\limits_{0}^{R}\sin(t\xi)\,E_{s}^{-}(\xi,a)\,d\xi.
\end{align}
\end{enumerate}
For fixed \(t\in(0,\infty)\), the limits  exist uniformly with respect to \(a\),
from any compact subset of the strip \(0<\textup{Re}\,a<1\).

\end{theorem}
\begin{remark}
\label{agr}
In \eqref{GECoT} and \eqref{GESiT}, the values of the square roots \(\sqrt{\varkappa(a)}\)
and \(\sqrt{\varkappa(1-a)}\)
should be chosen such that their products equal \(1\).
\end{remark}
\begin{remark}
\label{half}
For \(a=\tfrac{1}{2}\) there is only one eigenfunction
\[E\big(t,\tfrac{1}{2}\big)=2t^{-\frac{1}{2}}.\]
\end{remark}
\begin{remark}
\label{Redcy}
Since
\begin{subequations}
\label{redun}
\begin{alignat}{2}
\label{redunc}
E_{c}^{+}(t,a)&=E_{c}^{+}(t,1-a),\quad&E_{c}^{-}(t,a)&=-E_{c}^{-}(t,1-a),\\
\label{reduns}
 E_{s}^{+}(t,a)&=E_{s}^{+}(t,1-a),\quad&E_{s}^{-}(t,a)&=-E_{s}^{-}(t,1-a),
 \end{alignat}
 \end{subequations}
 each eigenfunction appears in the family \(\{E_{c,s}^{\pm}(t,a)\}_{0<\textup{Re}\,a<1}\)
 twice. To avoid this redundancy, we should consider the family where only one of the
 points \(a\) or \(1-a\) appear.
\end{remark}

\noindent
\textbf{5.}\hspace{1.0ex}If \(0<\textup{Re}\,a<1\) and
 \(x(t)\) is any of the eigenfunctions of the form either \eqref{GECoT}, or \eqref{GESiT},
then the integral \(\int\limits_{\mathbb{R}_{+}}|x(t)|^2\) diverges. Thus, none of these
eigenfunctions belong to \(L^2(\mathbb{R}_{+})\). This integral diverges both at points
\(t=+0\) and at point \(t=+\infty\). However, this integral diverges variously for \(a\) with
\(\textup{Re}\,a=\frac{1}{2}\) and for \(a\) with \(\textup{Re}\,a\not=\frac{1}{2}\).
If \(\textup{Re}\,a=\frac{1}{2}\), then the integrals diverge \emph{logarithmically} both at
\(t=+0\) and at \(t=+\infty\). If \(\textup{Re}\,a\not=\frac{1}{2}\), then the integrals
diverge more strongly: \emph{powerwise}. We try to construct eigenfunctions
of the operator \(\boldsymbol{\mathscr{C}}\) (of operator \(\boldsymbol{\mathscr{S}}\)) from \(L^2\) as continuous combinations of the eigenfunctions of the form \eqref{GECoT}
(of the form \eqref{GESiT}). Our hope is that singularities
of  "\emph{continuous linear combinations}" of
eigenfunctions, which are in some sense an \emph{averaging}  of eigenfunctions of the family,
are weaker than singularities of individual eigenfunctions. Such continuous linear combinations
should \emph{not} include eigenfunctions of the form \eqref{GECoT} and \eqref{GESiT}
with \(a\!:\,\textup{Re}\,a\not=\frac{1}{2}\). Singularities of eigenfunctions
with \(a\!:\,\textup{Re}\,a\not=\frac{1}{2}\) are too strong
and can not disappear by averaging. Thus, we have to restrict ourselves to \(a\)'s of the
form \(a=\frac{1}{2}+i\tau,\,\tau\in\mathbb{R}\).

Considering the case \(\textup{Re}\,a=\frac{1}{2}\) in more detail, we introduce
special notation for the eigenfunctions \(E_{c,s}^{\pm}(t,\frac{1}{2}+i\tau)\}\):
\begin{subequations}
\label{eFE}
\begin{alignat}{2}
\label{eFEc}
e_c^{+}(t,\tau)&=\tfrac{1}{2\sqrt{\pi}}E_{c}^{+}(t,\tfrac{1}{2}+i\tau),
\quad&e_c^{-}(t,\tau)&=\tfrac{1}{2i\sqrt{\pi}}E_{c}^{-}(t,\tfrac{1}{2}+i\tau),\\
\label{eFEs}
e_s^{+}(t,\tau)&=\tfrac{1}{2\sqrt{\pi}}E_{s}^{+}(t,\tfrac{1}{2}+i\tau),
\quad&e_s^{-}(t,\tau)&=\tfrac{1}{2i\sqrt{\pi}}E_{s}^{-}(t,\tfrac{1}{2}+i\tau).
\end{alignat}
\end{subequations}
(We include the normalizing factor \(\tfrac{1}{2\sqrt{\pi}}\) in the definition of the functions
\(e_{c,s}^{\pm}\).)
According to \eqref{CLCE}, \eqref{GECoT}, the functions \(e_{c,s}^{\pm}(t,\tau)\)
can be expressed as
\begin{subequations}
\label{ec}
\begin{align}
\label{ecp}
e_c^{+}(t,\tau)&=\tfrac{1}{2\sqrt{\pi}}\Big(
t^{-\frac{1}{2}-i\tau}\,c(\tau)+
t^{-\frac{1}{2}+i\tau}\,c(-\tau)\Big),\\
\label{ecm}
e_c^{-}(t,\tau)&=\tfrac{1}{2i\sqrt{\pi}}\Big(
t^{-\frac{1}{2}-i\tau}\,c(\tau)-
t^{-\frac{1}{2}+i\tau}\,c(-\tau)\Big),
\end{align}
\end{subequations}
\begin{subequations}
\label{es}
\begin{align}
\label{esp}
e_s^{+}(t,\tau)&=\tfrac{1}{2\sqrt{\pi}}\Big(
t^{-\frac{1}{2}-i\tau}\,s(\tau)+
t^{-\frac{1}{2}+i\tau}\,s(-\tau)\Big),\\
\label{esm}
e_s^{-}(t,\tau)&=\tfrac{1}{2i\sqrt{\pi}}\Big(
t^{-\frac{1}{2}-i\tau}\,s(\tau)-
t^{-\frac{1}{2}+i\tau}\,s(-\tau)\Big),
\end{align}
\end{subequations}
where \(c(\tau),\,s(\tau)\) are "phase factors":
\begin{subequations}
\label{etau}
\begin{align}
\label{etauc}
c(\tau)&=
2^{i\frac{\tau}{2}}\exp\big\{i\arg\Gamma(\tfrac{1}{4}+i\tfrac{\tau}{2})\big\},
\ \ -\infty<\tau<\infty,\\
\label{etaus}
s(\tau)&=2^{i\frac{\tau}{2}}\exp\big\{i\arg\Gamma(\tfrac{3}{4}+i\tfrac{\tau}{2})\big\},
\ \ -\infty<\tau<\infty.
\end{align}
\end{subequations}
In \eqref{etau},\,
\(\exp\{i\arg\Gamma(\zeta)\}=\tfrac{\Gamma(\zeta)}{|\Gamma(\zeta)|}\).
{\ }\\

\noindent
 Since \(c(\tau)=\overline{c(-\tau)},\,s(\tau)=\overline{s(-\tau)}\) for real \(\tau\),
  the values of the functions \(e_c^{+}(t,\tau)\),
  \(e_c^{-}(t,\tau)\), \(e_s^{+}(t,\tau)\), \(e_s^{-}(t,\tau)\) are \emph{real} for \(t\in(0,\infty)\),  \( \tau\in(0,\infty)\).

 \begin{remark}
 \label{posi}
 The parameter \(\tau\), which enumerates
 the families \(\{e_{c}^{\pm}(t,\tau\}\), \(\{e_{s}^{\pm}(t,\tau\}\), runs over
 the interval \((0,\infty)\). There is no need to consider negative \(\tau\). \textup{(}See \textup{Remark \ref{Redcy}}\textup{)}.
  \end{remark}

\noindent
\textbf{6.}\hspace{1.0ex} Let us introduce four integral transforms \(\boldsymbol{\mathscr{T}}_c^{+}\),
\(\boldsymbol{\mathscr{T}}_c^{-}\), \(\boldsymbol{\mathscr{T}}_s^{+}\), \(\boldsymbol{\mathscr{T}}_s^{-}\).\\  For
 \(\phi(t)\in{}L^1(\mathbb{R}_{+})\) and \(t>0\), let
 us define
 \begin{subequations}
 \label{MIT}
 \begin{alignat}{2}
 \label{MITp}
 (\boldsymbol{\mathscr{T}}_c^{+}\phi)(t)&=
 \int\limits_{\mathbb{R}_{+}}\,e_{c}^{+}(t,\tau)\phi(\tau)\,d\tau,\quad&
 (\boldsymbol{\mathscr{T}}_c^{-}\phi)(t)&=
 \int\limits_{\mathbb{R}_{+}}\,e_{c}^{-}(t,\tau)\phi(\tau)\,d\tau,\\
  \label{MITm}
(\boldsymbol{\mathscr{T}}_s^{+}\phi)(t)&=
\int\limits_{\mathbb{R}_{+}}\,e_{s}^{+}(t,\tau)\phi(\tau)\,d\tau,\quad&
 (\boldsymbol{\mathscr{T}}_s^{-}\phi)(t)&=
 \int\limits_{\mathbb{R}_{+}}\,e_{s}^{-}(t,\tau)\phi(\tau)\,d\tau,
 \end{alignat}
 \end{subequations}
 \begin{lemma}
 \label{PrTrT}
 If \(\phi(\tau)\in{}L^1(\mathbb{R}_{+})\), and \(x(t)=(\boldsymbol{\mathscr{T}\phi)}(t)\),
 where \(\boldsymbol{\mathscr{T}}\) is any of the above--introduced four transformations
\(\boldsymbol{\mathscr{T}}_{c,s}^{\pm}\), then the function \(x(t)\) is continuous on the
interval \((0,\infty)\) and the estimate
\begin{equation}
\label{EstL1}
|x(t)|\leq\tfrac{1}{\sqrt{\pi}}\|\phi\|_{_{L^1(\mathbb{R}_{+})}}\cdot{}t^{-\frac{1}{2}},\quad 0<t<\infty,
\end{equation}
holds.
 \end{lemma}
\begin{proof}
Let  \(e(t,\tau)\) be any of the four above--introduced functions \(e_{c}^{+}(t,\tau)\),
\(e_{c}^{-}(t,\tau)\), \(e_{s}^{+}(t,\tau)\),
\(e_{s}^{-}(t,\tau)\).
The function \(e(t,\tau)\) is continuous with respect to \(t\) at each \(t>0,\tau>0\) and satisfies
the estimate
\begin{equation}
\label{EstFe}
|e(t,\tau|\leq\tfrac{1}{\sqrt{\pi}}t^{-\frac{1}{2}},\quad 0<t<\infty,\,0<\tau<\infty.
\end{equation}
Now Lemma \ref{PrTrT} is a consequence of standard results of Lebesgue integration theory.
\end{proof}
\begin{theorem}
\label{EFBS}
Let \(\phi(\tau)\) be a  function,
satisfying the condition
\begin{equation}
\label{RsCoph}
\int\limits_{0}^{\infty}|\phi(\tau)|e^{\frac{\pi}{2}\tau}\,d\tau<\infty.
\end{equation}
and
\begin{alignat}{2}
\label{ATto}
x_{c}^{+}(t)&=(\boldsymbol{\mathscr{T}}_{c}^{+}\phi)(t), \quad & x_{c}^{-}(t)&=(\boldsymbol{\mathscr{T}}_{c}^{-}\phi)(t),\\
x_{s}^{+}(t)&=(\boldsymbol{\mathscr{T}}_{s}^{+}\phi)(t), \quad & x_{s}^{-}(t)&=(\boldsymbol{\mathscr{T}}_{c}^{-}\phi)(t),
\end{alignat}
Then the functions \(x_{c}^{+}(t),\,x_{c}^{-}(t)\) are eigenfunctions \textup{(}in the broad sense\textup{)} of the cosine transform \(\boldsymbol{\mathscr{C}}\) and the functions
\(x_{s}^{+}(t),\,x_{s}^{-}(t)\) are eigenfunctions \textup{(}in the broad sense\textup{)} of the sine transform \(\boldsymbol{\mathscr{S}}\), i.e.
\begin{subequations}
\label{EFiBSc}
\begin{align}
\label{EFiBScp}
x_{c}^{+}(t)&=\phantom{-}
\lim_{R\to\infty}\sqrt{\tfrac{2}{\pi}}\int\limits_{0}^{R}\cos(t\xi)\,x_{c}^{+}(\xi)\,d\xi,\\
\label{EFiBScm}
x_{c}^{-}(t)&=-
\lim_{R\to\infty}\sqrt{\tfrac{2}{\pi}}\int\limits_{0}^{R}\cos(t\xi)\,x_{c}^{-}(\xi)\,d\xi\,.
\end{align}
\end{subequations}
\vspace{-2.0ex}
and
\vspace{-2.0ex}
\begin{subequations}
\label{EFiBSs}
\begin{align}
\label{EFiBSsp}
x_{s}^{+}(t)&=\phantom{-}
\lim_{R\to\infty}\sqrt{\tfrac{2}{\pi}}\int\limits_{0}^{R}\sin(t\xi)\,x_{s}^{+}(\xi)\,d\xi,\\
\label{EFiBSsm}
x_{s}^{-}(t)&=-
\lim_{R\to\infty}\sqrt{\tfrac{2}{\pi}}\int\limits_{0}^{R}\sin(t\xi)\,x_{s}^{-}(\xi)\,d\xi\,.
\end{align}
\end{subequations}
for every \(t\in(0,\infty)\).
In particular, in \eqref{EFiBSc}, \eqref{EFiBSs}   the limits  exist.
\end{theorem}
\begin{proof} According to Theorem \ref{GeEig} and \eqref{eFE},
\begin{equation*}
e_{c}^{+}(t,\tau)=\lim_{R\to\infty}\sqrt{\tfrac{2}{\pi}}\int\limits_{0}^{R}\cos(t\xi)e_{c}^{+}(\xi,\tau)\,d\xi\ \
\textup{for every}\ t,\,\tau.
\end{equation*}
Multiplying by \(\phi(\tau)\) and integrating with respect to \(\tau\), we obtain
\begin{equation*}
x_{c}^{+}(t)=\sqrt{\tfrac{2}{\pi}}%
\int\limits_{0}^{\infty}\bigg(\lim_{R\to\infty}%
\int\limits_{0}^{R}\cos(t\xi)e_{c}^{+}(\xi,\tau)\,d\xi\bigg)\phi(\tau)\,d\tau.
\end{equation*}
From \eqref{Ewrta} we obtain the estimate
\[\bigg|\int\limits_{0}^{R}\cos(t\xi)e_{c}^{+}(\xi,\tau)\,d\xi\bigg|\leq{}%
Ct^{-\frac{1}{2}}e^{\frac{\pi}{2}\tau},\ \
\forall\,\,R<\infty,\,\tau\in\mathbb{R}_+,\,t\in\mathbb{R}_+,\]
 where the value \(C<\infty\) does not depend on \(R,\,\tau,\,t\).
 This estimate and condition \eqref{RsCoph} for the function \(\phi(t)\) allow
us to apply the Lebesgue dominated convergence theorem:
\begin{multline}
\label{RPaL1}
\int\limits_{0}^{\infty}\bigg(\lim_{R\to\infty}
\int\limits_{0}^{R}\cos(t\xi)e_{c}^{+}(\xi,\tau)\,d\xi\bigg)\phi(\tau)\,d\tau=\\
\lim_{R\to\infty}%
\int\limits_{0}^{\infty}\bigg(
\int\limits_{0}^{R}\cos(t\xi)e_{c}^{+}(\xi,\tau)\,d\xi\bigg)\phi(\tau)\,d\tau.
\end{multline}
Thus,
\[x_{c}^{+}(t)=\lim_{R\to\infty}\sqrt{\tfrac{2}{\pi}}%
\int\limits_{0}^{\infty}\bigg(
\int\limits_{0}^{R}\cos(t\xi)e_{c}^{+}(\xi,\tau)\,d\xi\bigg)\phi(\tau)\,d\tau.\]
On the other hand, using the estimate \eqref{EstFe} for \(e_{c}^{+}(\xi,\tau)\), we can justify the change of order of integration in the series integral which appears on the right--hand side of the above equality. For any finite \(R\),
\begin{multline*}
\int\limits_{0}^{\infty}\bigg(
\int\limits_{0}^{R}\cos(t\xi)e_{c}^{+}(\xi,\tau)\,d\xi\bigg)\phi(\tau)\,d\tau=\\
\int\limits_{0}^{R}\cos(t\xi)%
\bigg(\int\limits_{0}^{\infty}e_{c}^{+}(\xi,\tau)\phi(\tau)\,d\tau\bigg)\,d\xi=
\int\limits_{0}^{R}\cos(t\xi)x_{c}^{+}(\xi)\,d\xi.
\end{multline*}
Finally, we obtain the equality
\begin{math}
x_{c}^{+}(t)=\lim_{R\to\infty}\int\limits_{0}^{R}\cos(t\xi)x_{c}^{+}(\xi)\,d\xi,
\end{math}
i.e. the equality \eqref{EFiBScp} for the function \(x_{c}^{+}\). The equality
\eqref{EFiBScm} for the function \(x_{c}^{-}\) and the equalities \eqref{EFiBSs} for the functions \(x_{s}^{+},\,x_{s}^{+}\) can be obtained analogously.
\end{proof}
\begin{remark}
\label{WhF}
In \textup{Theorem \ref{EFBS}} we assume that the function \(\phi\) satisfies condition \eqref{RsCoph}.
Assuming only that \(\int\limits_{0}^{\infty}|\phi(\tau)|\,d\tau<\infty\), we can not justify
the equality \eqref{RPaL1}.
 To apply the Lebesgue dominated convergence theorem,
we need the estimate
\[\sup_{\substack{R\in(0,\infty)\\{}\tau\in(-\infty,\infty)}}
\bigg|\int\limits_{0}^{R}(\cos\xi)\cdot\xi^{-\tfrac{1}{2}+i\tau}\,d\xi\bigg|<\infty .\]
We are, however, able to establish \eqref{Ewrta}, but this estimate is not strong enough.

\textsf{\textsl{The question}} of whether the equalities \eqref{EFiBSc}, \eqref{EFiBSs} hold under the assumption
\(\int\limits_{0}^{\infty}|\phi(\tau)|\,d\tau<\infty\) \textsl{\textsf{remains open.}}
\end{remark}

\noindent
\textbf{7.}\hspace{1.0ex} Our considerations in the context of \(L^2\)-theory on the operators
\(\boldsymbol{\mathscr{C}}\) and \(\boldsymbol{\mathscr{S}}\) are based on \(L^2\)-theory for the Melline transform.
(See the article "Melline Transform" on page 192 of \cite[Volume 6]{7} and
references there.)
 The Melline transform \(\boldsymbol{\mathscr{M}}\) is defined by
 \begin{equation*}
 (\boldsymbol{\mathscr{M}}f)(\zeta)=\int\limits_{0}^{\infty}f(t)t^{\zeta-1}\,dt.
 \end{equation*}
If the function \(f(t)\in{}L^2(\mathbb{R}_{+})\) is \emph{compactly} supported in the
 \emph{open} interval \((0,\infty)\), then the function \(\Phi(\zeta)= (\boldsymbol{\mathscr{M}}f)(\zeta)\) of variable \(\zeta\) is defined in the whole complex
 \(\zeta\)-plane and holomorphic there. The function \(f(t)\) can be recovered from the
 function  \(\Phi= \boldsymbol{\mathscr{M}}f\) by the formula
 \[f(t)=\tfrac{1}{2\pi{}i}\int\limits_{\textup{Re}\,\zeta=c}\Phi(\zeta)\,t^{-\zeta}\,d\zeta,\]
 where \(c\) is an arbitrary real number. Moreover, the Parseval equality
 \[\int\limits_{0}^{\infty}|f(t)|^2dt=
 \frac{1}{2\pi}\int\limits_{\textup{Re}\,\zeta=\frac{1}{2}}|\Phi(\zeta)|^{2}\,|d\zeta|\]
 holds (from which we recognize the significance of the vertical line \(\textup{Re}\,\zeta=\tfrac{1}{2}\)).
 Thus the Melline transform \(\boldsymbol{\mathscr{M}}\) generates the linear operator defined
 on the set of all compactly supported functions \(f\) from \(L^2(\mathbb{R}_{+})\) which
 maps this set isometrically into the space \(L^2\big(\textup{Re}\,\zeta=\tfrac{1}{2}\big)\) of
 functions defined on the vertical line \(\textup{Re}\,\zeta=\tfrac{1}{2}\) and which are square--integrable there. Since the set of all compactly supported functions \(f\) is dense in \(L^2(\mathbb{R}_{+})\),
 this operator can be extended to an isometrical operator defined on the
 whole  \(L^2(\mathbb{R}_{+})\) which maps \(L^2(\mathbb{R}_{+})\) isometrically into \(L^2\big(\textup{Re}\,\zeta=\tfrac{1}{2}\big)\). We will continue to denote this extended operator by \(\boldsymbol{\mathscr{M}}\).

 It turns out that the operator \(\boldsymbol{\mathscr{M}}\) maps
 the space \(L^2(\mathbb{R}_{+})\) \emph{onto} the whole space   \(L^2\big(\textup{Re}\,\zeta=\tfrac{1}{2}\big)\).
  The inverse operator \(\boldsymbol{\mathscr{M}}^{-1}\) is defined \emph{everywhere} on \(L^2\big(\textup{Re}\,\zeta=\tfrac{1}{2}\big)\). If \(\Phi\in{}L^2\big(\textup{Re}\,\zeta=\tfrac{1}{2}\big)\),
then the function \[f(t)=(\boldsymbol{\mathscr{M}}^{-1}\Phi)(t)\] is defined as an
\(L^2(\mathbb{R}_{+})\)-function and can be expressed as
\begin{subequations}
\label{Mel}
\begin{equation}
\label{MelI}
f(t)=\tfrac{1}{2\pi}\int_{-\infty}^{\infty}\Phi\big(\tfrac{1}{2}+i\tau\big)
\,t^{-\frac{1}{2}-i\tau}d\tau,\ \ 0<t<\infty.
\end{equation}
Furthermore, the function
\[\Phi\big(\tfrac{1}{2}+i\tau\big)=(\boldsymbol{\mathscr{M}}f)(\tfrac{1}{2}+i\tau\big)\]
can be expressed as
 \begin{equation}
 \label{MelD}
 \Phi\big(\tfrac{1}{2}+i\tau\big)=\int\limits_{0}^{\infty}f(t)\,t^{-\frac{1}{2}+i\tau}dt,
 \ \ -\infty<\tau<\infty.
 \end{equation}
 The pair of formulas \eqref{MelI} and \eqref{MelD} together with the Parseval equality
 \begin{equation}
 \label{MelP}
 \int\limits_{0}^{\infty}|f(t)|^2dt=\tfrac{1}{2\pi}\int\limits_{-\infty}^{\infty}
 \big|\Phi\big(\tfrac{1}{2}+i\tau\big)\big|^2\,d\tau
 \end{equation}
 make up the most import part of the \(L^2\)-theory of Melline transform.
\end{subequations}{\ }\\

\noindent
\textbf{8.} Developing \(L^2\)-theory of the cosine and sine transforms, we first of
all prove
 \begin{lemma}
 \label{BMO}
 Let \(\phi(t)\in{}L^1(\mathbb{R}_{+})\cap{}L^2(\mathbb{R}_{+})\).
 Then
 \begin{equation}
\label{L2Co}
\int\limits_{\mathbb{R}_{+}}|(\boldsymbol{\mathscr{T}}\phi)(t)|^2dt=
\int\limits_{\mathbb{R}_{+}}|\phi(\tau)|^2d\tau\,,
\end{equation}
where \(\boldsymbol{\mathscr{T}}\) is any of the above--introduced \textup{(}see \eqref{MIT}\textup{)} four transformations
\(\boldsymbol{\mathscr{T}}_{c,s}^{\pm}\).
 \end{lemma}
\begin{proof}
The proof is based on the Parseval equality  for the Melline transform.
We present the transformations \(\boldsymbol{\mathscr{T}}_{c,s}^{\pm}\) as inverse Melline
transforms. Given a function \(\phi(\tau)\) defined for \(\tau\in(0,\infty)\), we
introduce the functions
\begin{subequations}
\label{RedToMelc}
\begin{alignat}{2}
\label{RedToMelcp}
\Phi_{c}^{+}\big(\tfrac{1}{2}+i\tau\big)&=&\phantom{\textup{sign}\,(\tau)}&
\sqrt{\pi}\,c(\tau)\,\phi(|\tau|),\\
\label{RedToMelcm}
\Phi_{c}^{-}\big(\tfrac{1}{2}+i\tau\big)&=&\tfrac{1}{i}\,\textup{sign}\,(\tau)&
\sqrt{\pi}\,c(\tau)\,\phi(|\tau|),
\end{alignat}
\end{subequations}
and
\begin{subequations}
\label{RedToMels}
\begin{alignat}{2}
\label{RedToMelsp}
\Phi_{s}^{+}\big(\tfrac{1}{2}+i\tau\big)&=&\phantom{\textup{sign}\,(\tau)}&
\sqrt{\pi}\,s(\tau)\,\phi(|\tau|),\\
\label{RedToMelsm}
\Phi_{s}^{-}\big(\tfrac{1}{2}+i\tau\big)&=&\tfrac{1}{i}\,\textup{sign}\,(\tau)&
\sqrt{\pi}\,s(\tau)\,\phi(|\tau|),
\end{alignat}
\end{subequations}
\emph{which  are defined for} \(\tau\in(-\infty,\infty)\). Here \(c(\tau),\,s(\tau)\) are
the "phase factors" introduced in \eqref{etau}.
It is clear that \(\big|\Phi\big(\tfrac{1}{2}+i\tau\big)\big|=\sqrt{\pi}\phi(|\tau)|\), thus
\[\int\limits_{-\infty}^{\infty}\big|\Phi\big(\tfrac{1}{2}+i\tau\big)\big|^2\,d\tau=
2\pi\int\limits_{0}^{\infty}|\phi(\tau)|^2d\tau,\]
where \(\Phi\) is any of the four functions \(\Phi_{c}^{+}\), \(\Phi_{c}^{-}\),
\(\Phi_{s}^{+}\), \(\Phi_{s}^{-}\). Comparing \eqref{MITp}, \eqref{ecp} and \eqref{RedToMelcp},
we see that the function \((\boldsymbol{\mathscr{T}}_{c}^{+}\phi)(t)\)
can be interpreted as the inverse Melline transform of the function \(\Phi_{c}^{+}\):
\begin{subequations}
\label{Inte}
\begin{equation}
\label{Intecp}
(\boldsymbol{\mathscr{T}}_{c}^{+}\phi)(t)=
\tfrac{1}{2\pi}\int\limits_{-\infty}^{\infty}
t^{-\frac{1}{2}-i\tau}\Phi_{c}^{+}\big(\tfrac{1}{2}+i\tau\big)\,d\tau.
\end{equation}
The Parseval equality transform,
as applied to the inverse Melline transform of the function \(\varphi_{c}^{+}(\tau)\), yields:
\begin{equation*}
\int\limits_{0}^{\infty}|(\boldsymbol{\mathscr{T}}_{c}^{+}\phi)(t)|^2dt
=\tfrac{1}{2\pi}\int\limits_{-\infty}^{\infty}
\big|\Phi_{c}^{+}\big(\tfrac{1}{2}+i\tau\big)\big|^2d\tau=
\int\limits_{0}^{\infty}|\phi(\tau)|^2d\tau.
\end{equation*}
This is equality \eqref{L2Co} for the transform \(\boldsymbol{\mathscr{T}}_{c}^{+}\).

The functions \(\boldsymbol{\mathscr{T}}_{c}^{-}\phi\),  \(\boldsymbol{\mathscr{T}}_{s}^{+}\phi\),
 \(\boldsymbol{\mathscr{T}}_{s}^{-}\phi\) can also be interpreted as inverse Melline transforms:
\begin{equation}
\label{Intecm}
(\boldsymbol{\mathscr{T}}_{c}^{-}\phi)(t)=
\tfrac{1}{2\pi}\int\limits_{-\infty}^{\infty}
t^{-\frac{1}{2}-i\tau}\Phi_{c}^{-}\big(\tfrac{1}{2}+i\tau\big)\,d\tau,
\end{equation}
\end{subequations}
and
\begin{subequations}
\label{Intes}
\begin{align}
\label{Intesp}
(\boldsymbol{\mathscr{T}}_{s}^{+}\phi)(t)=
\tfrac{1}{2\pi}\int\limits_{-\infty}^{\infty}
t^{-\frac{1}{2}-i\tau}\Phi_{s}^{+}\big(\tfrac{1}{2}+i\tau\big)\,d\tau,\\
\label{Intesm}
(\boldsymbol{\mathscr{T}}_{s}^{-}\phi)(t)=
\tfrac{1}{2\pi}\int\limits_{-\infty}^{\infty}
t^{-\frac{1}{2}-i\tau}\Phi_{s}^{-}\big(\tfrac{1}{2}+i\tau\big)\,d\tau.
\end{align}
\end{subequations}
The Parseval equalities, as applied to the inverse Melline transform of the functions
\(\Phi_{c}^{-}\), \(\Phi_{s}^{+}\) and \(\Phi_{s}^{-}\), yield
the equalities \eqref{L2Co} for the transforms \(\boldsymbol{\mathscr{T}}_{c}^{-}\),
\(\boldsymbol{\mathscr{T}}_{s}^{+}\) and \(\boldsymbol{\mathscr{T}}_{s}^{-}\), respectively.
\end{proof}

\noindent
\textbf{9.}\hspace{1.0ex}
According to Lemma \ref{BMO}, the operators \(\boldsymbol{\mathscr{T}}_{c}^{+}\),
\(\boldsymbol{\mathscr{T}}_{c}^{-}\), \(\boldsymbol{\mathscr{T}}_{s}^{+}\),
\(\boldsymbol{\mathscr{T}}_{s}^{-}\) are linear operators each of which is defined on the
linear manifold \(L^{1}(\mathbb{R}_{+})\cap{}L^{2}(\mathbb{R}_{+})\) of the Hilbert space \(L^{2}(\mathbb{R}_{+})\)
and which maps this linear manifold into \(L^{2}(\mathbb{R}_{+})\) \emph{isometrically}.
Since the set \(L^{1}(\mathbb{R}_{+})\cap{}L^{2}(\mathbb{R}_{+})\) is dense in \(L^{2}(\mathbb{R}_{+})\),   each of these operators can be extended to an operator
defined on the whole space  \(L^{2}(\mathbb{R}_{+})\),  which maps  \(L^{2}(\mathbb{R}_{+})\)
into \(L^{2}(\mathbb{R}_{+})\) isometrically. We will continue to write  \(\boldsymbol{\mathscr{T}}_{c}^{+}\),
\(\boldsymbol{\mathscr{T}}_{c}^{-}\), \(\boldsymbol{\mathscr{T}}_{s}^{+}\)
and \(\boldsymbol{\mathscr{T}}_{s}^{-}\) for the extended operators.

We now consider the operators \(\boldsymbol{\mathscr{T}}_{c}^{+}\),
\(\boldsymbol{\mathscr{T}}_{c}^{-}\), \(\boldsymbol{\mathscr{T}}_{s}^{+}\),
\(\boldsymbol{\mathscr{T}}_{s}^{-}\) as operators defined on \emph{all} of
\(L^{2}(\mathbb{R}_{+})\), mapping \(L^{2}(\mathbb{R}_{+})\) into \(L^{2}(\mathbb{R}_{+})\)
isometrically and acting on the functions \(\phi(t)\in{}L^{1}(\mathbb{R}_{+})\cap{}L^{2}(\mathbb{R}_{+})\) according to \eqref{MIT}.
\begin{theorem} {\ }\\[-3.0ex]
\label{RepEig}
\begin{enumerate}
\item[\textup{1}.]
The range of values of the operator \(\boldsymbol{\mathscr{T}}_{c}^{+}\) is
the eigensubspace \(\mathcal{C}_{+1}\) of the operator \(\boldsymbol{\mathscr{C}}\);
\item[\textup{2}.]
The range of values of the operator \(\boldsymbol{\mathscr{T}}_{c}^{-}\) is
the eigensubspace \(\mathcal{C}_{-1}\) of the operator \(\boldsymbol{\mathscr{C}}\);
\item[\textup{3}.]
The range of values of the operator \(\boldsymbol{\mathscr{T}}_{s}^{+}\) is
the eigensubspace \(\mathcal{S}_{+1}\) of the operator \(\boldsymbol{\mathscr{S}}\);
\item[\textup{4}.]
The range of values of the operator \(\boldsymbol{\mathscr{T}}_{s}^{-}\) is
the eigensubspace \(\mathcal{S}_{-1}\) of the operator \(\boldsymbol{\mathscr{S}}\).
\end{enumerate}
\end{theorem}

\begin{remark}
\label{HiFo}
Since the operators \(\boldsymbol{\mathscr{T}}_{c}^{+}\),\,\(\boldsymbol{\mathscr{T}}_{c}^{-}\),
\(\boldsymbol{\mathscr{T}}_{s}^{+}\),\,\(\boldsymbol{\mathscr{T}}_{s}^{-}\) act isometrically
from \(L^2(\mathbb{R}_{+})\) into \(L^2(\mathbb{R}_{+})\),
the equalities
\begin{subequations}
\label{EigC}
\begin{align}
\label{EigCp}
(\boldsymbol{\mathscr{T}}_{c}^{+})^{\ast}\boldsymbol{\mathscr{T}}_{c}^{+}=
\boldsymbol{\mathscr{I}},\quad\boldsymbol{\mathscr{T}}_{c}^{+}%
(\boldsymbol{\mathscr{T}}_{c}^{+})^{\ast}=\boldsymbol{\mathscr{P}}_{c}^{+},\quad%
\boldsymbol{\mathscr{C}}\boldsymbol{\mathscr{T}}_{c}^{+}=
\phantom{-}\boldsymbol{\mathscr{T}}_{c}^{+};\\
\label{EigCm}
(\boldsymbol{\mathscr{T}}_{c}^{-})^{\ast}\boldsymbol{\mathscr{T}}_{c}^{-}=
\boldsymbol{\mathscr{I}},\quad\boldsymbol{\mathscr{T}}_{c}^{-}%
(\boldsymbol{\mathscr{T}}_{c}^{-})^{\ast}=\boldsymbol{\mathscr{P}}_{c}^{-},\quad%
\boldsymbol{\mathscr{C}}\boldsymbol{\mathscr{T}}_{c}^{-}=-\boldsymbol{\mathscr{T}}_{c}^{-}.
\end{align}
\end{subequations}
and
\begin{subequations}
\label{EigS}
\begin{align}
\label{EigSp}
(\boldsymbol{\mathscr{T}}_{s}^{+})^{\ast}\boldsymbol{\mathscr{T}}_{s}^{+}=
\boldsymbol{\mathscr{I}},\quad\boldsymbol{\mathscr{T}}_{s}^{+}%
(\boldsymbol{\mathscr{T}}_{s}^{+})^{\ast}=\boldsymbol{\mathscr{P}}_{s}^{+},\quad%
\boldsymbol{\mathscr{S}}\boldsymbol{\mathscr{T}}_{s}^{+}=
\phantom{-}\boldsymbol{\mathscr{T}}_{s}^{+};\\
\label{EigSm}
(\boldsymbol{\mathscr{T}}_{s}^{-})^{\ast}\boldsymbol{\mathscr{T}}_{s}^{-}=
\boldsymbol{\mathscr{I}},\quad\boldsymbol{\mathscr{T}}_{s}^{-}%
(\boldsymbol{\mathscr{T}}_{s}^{-})^{\ast}=\boldsymbol{\mathscr{P}}_{s}^{-},\quad%
\boldsymbol{\mathscr{S}}\boldsymbol{\mathscr{T}}_{s}^{-}=-\boldsymbol{\mathscr{T}}_{s}^{-}.
\end{align}
\end{subequations}
hold, where \(\boldsymbol{\mathscr{P}}_{c}^{+},\,\boldsymbol{\mathscr{P}}_{c}^{-}\), \(\boldsymbol{\mathscr{P}}_{s}^{+}\) and \(\boldsymbol{\mathscr{P}}_{s}^{-}\) are
orthogonal projectors from \(L^2(\mathbb{R})_{+}\) onto the eigensubspaces \(\mathcal{C}_{+1}\),
\(\mathcal{C}_{-1}\),\,\(\mathcal{S}_{+1}\) and
\(\mathcal{S}_{-1}\), respectively
and \((\boldsymbol{\mathscr{T}}_{c}^{+})^{\ast}\), \((\boldsymbol{\mathscr{T}}_{c}^{-})^{\ast}\),
 \((\boldsymbol{\mathscr{T}}_{s}^{+})^{\ast}\), \((\boldsymbol{\mathscr{T}}_{s}^{-})^{\ast}\)
are the operators Hermitian--conjugated to the operators
\((\boldsymbol{\mathscr{T}}_{c}^{+}), \,(\boldsymbol{\mathscr{T}}_{c}^{-})\),
\(\boldsymbol{\mathscr{T}}_{s}^{+}), \,(\boldsymbol{\mathscr{T}}_{s}^{-})\) with respect to the
standard scalar product in the Hilbert space \(L^2(\mathbb{R}_{+})\).

In particular, the operators \((\boldsymbol{\mathscr{T}}_{c}^{+})^{\ast}\),
\((\boldsymbol{\mathscr{T}}_{c}^{-})^{\ast}\), \((\boldsymbol{\mathscr{T}}_{s}^{+})^{\ast}\) and
\((\boldsymbol{\mathscr{T}}_{s}^{-})^{\ast}\) are generalized inverses\,%
\footnote{\,In the sense of Moore-Penrose, for example.}
 of the operators
\(\boldsymbol{\mathscr{T}}_{c}^{+}\),
\(\boldsymbol{\mathscr{T}}_{c}^{-}\), \(\boldsymbol{\mathscr{T}}_{s}^{+}\) and
\(\boldsymbol{\mathscr{T}}_{s}^{-}\), respectively.
\end{remark}

\noindent
 It with mentioning that
\begin{subequations}
 \label{CoMIT}
 \begin{alignat}{2}
 \label{CoMITp}
\big((\boldsymbol{\mathscr{T}}_c^{+})^{\ast}\,x\big)(\tau)&=
 \int\limits_{\mathbb{R}_{+}}\,e_{c}^{+}(t,\tau)x(t)\,dt,\ \ &
\big((\boldsymbol{\mathscr{T}}_c^{-})^{\ast}x\big)(\tau)&=
 \int\limits_{\mathbb{R}_{+}}\,e_{c}^{-}(t,\tau)x(t)\,dt,\\
  \label{CoMITm}
\big((\boldsymbol{\mathscr{T}}_s^{+})^{\ast}x\big)(\tau)&=
\int\limits_{\mathbb{R}_{+}}\,e_{s}^{+}(t,\tau)\,x(t)\,dt,\ \ &
 \big((\boldsymbol{\mathscr{T}}_s^{-})^{\ast}x\big)(\tau)&=
 \int\limits_{\mathbb{R}_{+}}\,e_{s}^{-}(t,\tau)\,x(t)\,dt.
 \end{alignat}
 \end{subequations} {\ }\\

\noindent
Theorem \ref{RepEig} is a consequence of the following
\begin{lemma}
\label{Titch}
 Let a function \(x(t)\) belong to \(L^2(\mathbb{R}_{+})\) and \(\hat{x}_c(t)\)
 and \(\hat{x}_s(t)\) be the cosine and sine Fourier transform
 of the function \(x\):
 \begin{subequations}
 \label{FuT}
 \begin{gather}
 \label{FuTc}
\hat{x}_c(t)=\sqrt{\frac{2}{\pi}}\int\limits_{0}^{\infty}x(s)\,\cos
(ts)\,ds,\\[1.0ex]
\label{FuTs}
\hat{x}_s(t)=\sqrt{\frac{2}{\pi}}\int\limits_{0}^{\infty}x(s)\,\sin
(ts)\,ds
 \end{gather}
 \end{subequations}
Let \(\Phi_{x}(\zeta),\,\Phi_{\hat{x}_c}(\zeta)\) and
\(\Phi_{\hat{x}_s}(\zeta)\) be the Melline transforms of the
functions \(x,\,\hat{x}_c\) and \(\hat{x}_s\), respectively.
\textup{(}All three functions \(x,\,\hat{x}_c,\,\hat{x}_s\) belong
to \(L^2(0,\infty)\), so their Melline transforms exist and are
\(L^2\) functions on the vertical line
\(\textup{Re}\,\zeta=\frac{1}{2}\).\textup{)}

Then for \(\zeta:\,\textup{Re}\,\zeta=\frac{1}{2}\), the
equalities
\begin{subequations}
 \label{MMT}
\begin{gather}
\label{MMTc}%
 \Phi_{\hat{x}_c}(\zeta)=\Phi_{x}(1-\zeta)\cdot
2^{\,\zeta-\frac{1}{2}}%
\frac{\Gamma\big(\frac{\zeta}{2}\big)}{\Gamma\big(\frac{1}{2}-\frac{\zeta}{2}\big)}\,,\\[1.0ex]
\label{MMTs}
 \Phi_{\hat{x}_s}(\zeta)=\Phi_{x}(1-\zeta)\cdot
2^{\,\zeta-\frac{1}{2}}%
\frac{\Gamma\big(\frac{1}{2}+\frac{\zeta}{2}\big)}{\Gamma\big(1-\frac{\zeta}{2}\big)}\,.
\end{gather}
\end{subequations}
hold.
\end{lemma}
\begin{proof}
It is enough to prove the equalities \eqref{MMT} assuming that the functions \(x(t),
\hat{x}_c(t), \hat{x}_s(t)\) are continuous and belong to \(L^{2}(\mathbb{R}_{+})\cap{}L^{1}(\mathbb{R}_{+})\):
the set of such functions \(x\) is dense in \(L^2(\mathbb{R})\) and all three transforms, cosine,  sine  and Melline transforms, act continuously from \(L^2\) to \(L^2\).
Under these extra assumptions on the functions \(x(t),\hat{x}_c(t), \hat{x}_s(t)\),
the Melline transforms \(\Phi_{x}(\zeta),\,\Phi_{\hat{x}_c}(\zeta)\), \(\Phi_{\hat{x}_c}(\zeta)\) are defined everywhere on the vertical line \(\textup{Re}\,\zeta=\frac{1}{2}\) and are continuous functions there. For such \(x\), the equalities \eqref{MMT} will be established for every
\(\zeta\! :\,\textup{Re}\,\zeta=\tfrac{1}{2}\).

We fix \(\zeta\! :\,\textup{Re}\,\zeta=\tfrac{1}{2}\). The Melline transform \(\Phi_{\hat{x}_c}(\zeta)\) is:
\[\Phi_{\hat{x}_c}(\zeta)=%
 \lim_{R\to\infty}\int\limits_{0}^{R}\hat{x}_c(t)t^{\zeta-1}dt.\]
 Substituting the expression \eqref{FuTc} for \(\hat{x}_c(t)\) into the last formula, we obtain:
 \begin{equation}
\label{T1}%
\Phi_{\hat{x}_c}(\zeta)=%
 \lim_{R\to\infty}\int\limits_{0}^{R}
\bigg(
\sqrt{\frac{2}{\pi}}\int\limits_{0}^{\infty}x(s)\,\cos(ts)\,ds\bigg)
t^{\zeta-1}\,dt.
\end{equation}
For fixed finite \(R\), we change the order of integration:
\begin{equation*}
\int\limits_{0}^{R} \bigg(
\int\limits_{0}^{\infty}x(s)\,\cos(ts)\,ds\bigg)
t^{\zeta-1}\,dt=
\int\limits_{0}^{\infty}x(s)\bigg(%
\int\limits_{0}^{R}\cos(ts)\,t^{\zeta-1}\,dt
\bigg)\,ds\,.
\end{equation*}%
 The change of order of integration is
justified by Fubini's theorem. Changing the variable \(ts=\tau\), we
get
\[\int\limits_{0}^{R}\cos(ts)\,t^{\zeta-1}\,dt=
s^{-\zeta}\int\limits_{0}^{Rs}\cos(\tau)\,\tau^{\zeta-1}\,d\tau\,.\]
Thus%
\begin{multline}%
\label{T3}%
 \int\limits_{0}^{R} \bigg(
\sqrt{\frac{2}{\pi}}\int\limits_{0}^{\infty}x(s)\,\cos(ts)\,ds\bigg)
t^{\zeta-1}\,dt=\\%
=\int\limits_{0}^{\infty}x(s)s^{-\zeta}\bigg(\sqrt{\frac{2}{\pi}}%
\int\limits_{0}^{Rs}\cos(\tau)\,\tau^{\zeta-1}\,d\tau\bigg)\,ds\,.
\end{multline}%
According to Lemma \ref{CCI}, for every \(s>0\),
\begin{equation*}%
\lim_{R\to\infty}\int\limits_{0}^{Rs}\cos(\tau)\,\tau^{\zeta-1}\,d\tau=
\Big(\cos\,\frac{\pi}{2}\zeta\Big)\,\Gamma(\zeta)\,,
\end{equation*}%
 The value \(\displaystyle\int\limits_{0}^{\rho}(\cos\tau)\tau^{\zeta-1}\,d\tau\),
 considered as a function of \(\rho\), vanishes at \(\rho=0\),
 is a continuous function of \(\rho\), and has a finite limit as
 \(\rho\to\infty\). Therefore there exist a finite
 \(M(\zeta)<\infty\) such that
 the estimate holds:
 \(\Big|\int\limits_{0}^{\rho}(\cos\tau)\tau^{\zeta-1}\,d\tau\Big|\leq{}M(\zeta)\),
 where the value \(M(\zeta)\) does not depend on~\(\rho\). In other
 words,
 \begin{equation*}
 \Bigg|\int\limits_{0}^{Rs}\cos(\tau)\,%
 \tau^{\zeta-1}\,d\tau\Bigg|\leq{}M(\zeta)<
 \infty\quad \forall\, s,R:\,0\leq{}s<\infty,\,0\leq{}R<\infty\,.
 \end{equation*}
 By the Lebesgue theorem on dominating convergence,
 \begin{multline}%
 \label{T4}
 \lim_{R\to\infty}\int\limits_{0}^{\infty}x(s)s^{-\zeta}\bigg(\sqrt{\frac{2}{\pi}}%
\int\limits_{0}^{Rs}\cos(\tau)\,\tau^{\zeta-1}\,d\tau\bigg)\,ds=\\
=\int\limits_{0}^{\infty}x(s)s^{-\zeta}\bigg(\sqrt{\frac{2}{\pi}}%
\int\limits_{0}^{\infty}\cos(\tau)\,\tau^{\zeta-1}\,d\tau\bigg)\,ds\,.
\end{multline}%
Taking into account the equalities \eqref{T1}, \eqref{T4} and using \eqref{CcI1} and \eqref{Gam1}, we reduce the last equality to the form
\[\Phi_{\hat{x}_c}(\zeta)=
\int\limits_{0}^{\infty}x(s)\,s^{-\zeta}\,ds\,\cdot\,
2^{\,\zeta-\frac{1}{2}}%
\frac{\Gamma\big(\frac{\zeta}{2}\big)}{\Gamma\big(\frac{1}{2}-\frac{\zeta}{2}\big)}\,.\]
To obtain \eqref{MMTc} from the previous equality, we need only consider that
\[\int\limits_{0}^{\infty}x(s)\,s^{-\zeta}\,ds=\Phi_{x}(1-\zeta)\,.\]
The equality  \eqref{MMTs} can be proved analogously.
\end{proof}
\begin{remark}
\label{phase}
The equalities \eqref{MMT} can be presented in the form
\begin{subequations}
\label{phn}
\begin{align}
\label{phnc}
 \Phi_{\hat{x}_c}\big(\tfrac{1}{2}+i\tau\big)&=\Phi_{x}\big(\tfrac{1}{2}-i\tau\big)\cdot
c^{\,2}(\tau),\\
\label{phns}
 \Phi_{\hat{x}_s}\big(\tfrac{1}{2}+i\tau\big)&=\Phi_{x}\big(\tfrac{1}{2}-i\tau\big)\cdot
s^{\,2}(\tau),
\end{align}
\end{subequations}
where \(c(\tau)\) and \(s(\tau)\) were introduced in \eqref{etau}.
\end{remark}
\hspace{1.0ex}\begin{proof}[Proof of \textup{Theorem \ref{RepEig}}]
Let \(x_c(t)\) be defined by \eqref{FuTc}.
The equality \(\boldsymbol{\mathscr{C}}x=x\), i.e. the equality
\(x_c(t)=x(t)\) for functions \(x_c(t),\,x(t)\), is equivalent to the equality
\[\Phi_{\hat{x}_{c}}\big(\tfrac{1}{2}+i\tau\big)=\Phi_{x}\big(\tfrac{1}{2}+i\tau\big)\]
for their Melline transforms. According to Lemma \ref{Titch}, \eqref{phnc},
the last equality can be reduced\,\footnote{\,Remember that \(c^{-1}(\tau)=c(-\tau).\)} %
to the form
\begin{subequations}
\label{Equc}
\begin{align}
\label{Equcp}
\Phi_{x}\big(\tfrac{1}{2}-i\tau\big)\cdot
c(\tau)&=\phantom{-}\Phi_{x}\big(\tfrac{1}{2}+i\tau\big)\cdot
c(-\tau),\quad -\infty<\tau<\infty.\\
\intertext{
Analogously, the equalities  \(\boldsymbol{\mathscr{C}}x=-x\),  \(\boldsymbol{\mathscr{S}}x=x\)
and  \(\boldsymbol{\mathscr{S}}x=-x\) for the functions \(x(t)\) are equivalent to the
equalities}
\label{Equcm}
\Phi_{x}\big(\tfrac{1}{2}-i\tau\big)\cdot
c(\tau)&=-\Phi_{x}\big(\tfrac{1}{2}+i\tau\big)\cdot
c(-\tau),\quad -\infty<\tau<\infty,
\end{align}
\end{subequations}
and
\begin{subequations}
\label{Equs}
\begin{align}
\label{Equsp}
\Phi_{x}\big(\tfrac{1}{2}-i\tau\big)\cdot
s(\tau)&=\phantom{-}\Phi_{x}\big(\tfrac{1}{2}+i\tau\big)\cdot
s(-\tau),\quad -\infty<\tau<\infty.\\
\label{Equsm}
\Phi_{x}\big(\tfrac{1}{2}-i\tau\big)\cdot
s(\tau)&=-\Phi_{x}\big(\tfrac{1}{2}+i\tau\big)\cdot
s(-\tau),\quad -\infty<\tau<\infty,
\end{align}
\end{subequations}
Thus each of the equalities
\(\boldsymbol{\mathscr{C}}x=x\), \(\boldsymbol{\mathscr{C}}x=-x\),
\(\boldsymbol{\mathscr{S}}x=x\), \(\boldsymbol{\mathscr{S}}x=-x\)
for the function \(x(t),\, 0<t<\infty\), is equivalent to the symmetry condition
 for its Melline transform \(\Phi_{x}\big(\tfrac{1}{2}+i\tau\big),\,-\infty<\tau<\infty\).
These symmetry conditions, which appear as conditions \eqref{Equc}, \eqref{Equs}, can be presented in the form
\begin{alignat*}{2}
\Phi_{x}\big(\tfrac{1}{2}+i\tau\big)&=& &\sqrt{\pi}\,c(\tau)\,\phi(|\tau|),
\quad -\infty<\tau<\infty,\\
\Phi_{x}\big(\tfrac{1}{2}+i\tau\big)&=&\,\tfrac{1}{i}\,\textup{sign}\,(\tau)\,&\sqrt{\pi}\,c(\tau)\,\phi(|\tau|),
 \quad -\infty<\tau<\infty,
\end{alignat*}
and
\begin{alignat*}{2}
\Phi_{x}\big(\tfrac{1}{2}+i\tau\big)&=& &\sqrt{\pi}\,s(\tau)\,\phi(|\tau|),
\quad -\infty<\tau<\infty,\\
\Phi_{x}\big(\tfrac{1}{2}+i\tau\big)&=&\,\tfrac{1}{i}\,\textup{sign}\,(\tau)\,&\sqrt{\pi}\,s(\tau)\,\phi(|\tau|),
 \quad -\infty<\tau<\infty,
\end{alignat*}
where \(\phi(\tau)\) is function defined for \(0<\tau<\infty\).
Comparing these expressions for the function \(\Phi_{x}\big(\tfrac{1}{2}+i\tau\big)\)
with the expressions \eqref{ec}, \eqref{es} for the eigenfunctions
\(e_{c}^{+}(t,\tau),\,e_{c}^{-}(t,\tau),\,e_{s}^{+}(t,\tau),\,e_{s}^{-}(t,\tau)\), we see
that in each of the four cases, the inversion formula
\[x(t)=\tfrac{1}{2\pi}\int\limits_{-\infty}^{\infty}%
t^{-\tfrac{1}{2}+i\tau}\Phi_{x}\big(\tfrac{1}{2}+i\tau\big)\,d\tau
\]
for the Melline transform can be presented in terms of the function \(\phi(\tau)\) as
\begin{subequations}
\label{REF}
\begin{align}
\label{REFc}
x(t)=\int\limits_{0}^{\infty}e_{c}^{+}(t,\tau)\,\phi(\tau)\,d\tau,\quad
x(t)=\int\limits_{0}^{\infty}e_{c}^{-}(t,\tau)\,\phi(\tau)\,d\tau,\\
\label{REFs}
x(t)=\int\limits_{0}^{\infty}e_{s}^{+}(t,\tau)\,\phi(\tau)\,d\tau,\quad
x(t)=\int\limits_{0}^{\infty}e_{s}^{-}(t,\tau)\,\phi(\tau)\,d\tau,
\end{align}
\end{subequations}
respectively.
Now the symmetries \eqref{Equc}, \eqref{Equs} of the function \(\Phi_{x}\big(\tfrac{1}{2}+i\tau\big)\)  are hidden in the structure
of functions \(e_{c}^{+}\), \(e_{c}^{-}\),
\(e_{s}^{+}\), \(e_{-}^{-}\).

Thus, the equalities  \(\boldsymbol{\mathscr{C}}x=x\), \(\boldsymbol{\mathscr{C}}x=-x\) and  \(\boldsymbol{\mathscr{S}}x=x\),
\(\boldsymbol{\mathscr{S}}x=-x\) for the functions \(x\) are equivalent to
representability of \(x\) in one of the four forms \eqref{REF},
i.e. in the form \(x=\boldsymbol{\mathscr{T}}_{c}^{+}\phi\),\,\(x=\boldsymbol{\mathscr{T}}_{c}^{-}\phi\),
\(x=\boldsymbol{\mathscr{T}}_{s}^{+}\phi\) and \(x=\boldsymbol{\mathscr{T}}_{s}^{-}\phi\),
respectively (with \(\phi\in{}L^2(\mathbb{R}_{+})\)).
\end{proof}

\noindent
\small
\textbf{Acknowledgements} I thank Armin Rahn for his careful reading of the manuscript
and his help in improving the English in this paper.
\normalsize


\end{document}